\pdfoutput=1
\documentclass{article}
\usepackage[utf8]{inputenc}
\usepackage[numbers]{natbib}
\usepackage{amsfonts,amsmath,amssymb,amsthm}
\usepackage{nicefrac,upgreek,mathtools,verbatim}
\usepackage[final]{hyperref}
\usepackage{geometry}[margin=1in]
\usepackage[mathscr]{eucal}
\usepackage{dsfont}
\usepackage[normalem]{ulem}
\usepackage{amsopn}

\newtheorem{proposition}{Proposition}[section]
\newtheorem{theorem}[proposition]{Theorem}

\newtheorem{definition}[proposition]{Definition}
\newtheorem{remark}[proposition]{Remark}
\newtheorem{lemma}[proposition]{Lemma}
\newtheorem{corollary}[proposition]{Corollary}
\newcommand{\Uadm}{\mathfrak U}
\newcommand{\Act}{\mathbb{U}}
\newcommand{\Usm}{\mathfrak U_{\mathsf{sm}}}

\newcommand{\Um}{\mathfrak U_{\mathsf{m}}}

\newcommand{\cC}{{\mathcal{C}}}   
 
\newcommand{\sF}{{\mathfrak{F}}}   
\newcommand{\cJ}{{\mathcal{J}}}  
  
\newcommand{\sL}{{\mathscr{L}}}  %
\newcommand{\Lp}{{L}}            

\newcommand{\cP}{{\mathcal{P}}}  

\newcommand{\RR}{\mathds{R}}

\newcommand{\Rd}{{\mathds{R}^{d}}}
\DeclareMathOperator{\Exp}{\mathbb{E}}

\newcommand{\D}{\mathrm{d}}

\newcommand{\Sob}{{\mathscr W}}    

\newcommand{\df}{:=}
\newcommand{\transp}{^{\mathsf{T}}}

\DeclareMathOperator*{\trace}{Tr}

\newcommand{\grad}{\nabla}

\newcommand{\abs}[1]{\lvert#1\rvert}
\usepackage{comment}
\usepackage{color}
\usepackage{xcolor}
\newcommand{\sy}[1]{{\color{blue} #1}}

\title{Robustness of Optimal Controlled Diffusions with Near-Brownian Noise via Rough Paths Theory}
\author{Somnath Pradhan, Zachary Selk and Serdar Y\"uksel }
\date{\today}

\begin{document}

\maketitle

\begin{abstract}
In many applications, typically only an ideal model is assumed, based on which an optimal control is designed and then applied to the actual system. This gives rise to the problem of performance loss due to the mismatch between the actual system and the assumed system. A particularly important case of this problem arises due to the Brownian idealization of the driving noise in controlled stochastic differential equations.
    In this article we present a robustness theorem for controlled stochastic differential equations driven by approximations of Brownian motion, where the approximations are those that converge to the Brownian under the rough paths topology along sample paths. These approximations include the Wong-Zakai, Karhunen-Lo\`eve, mollified Brownian and fractional Brownian processes, which can be interpreted to be more physical than the Brownian. We establish robustness using rough paths theory, which allows for a pathwise theory of stochastic differential equations. To this end, in particular, we show that within the class of Lipschitz continuous control policies, an optimal solution for the Brownian idealized model is near optimal for a true system driven by a non-Brownian (but near-Brownian) noise. 
\end{abstract}

\section{Introduction}

In stochastic control applications, typically only an ideal model is assumed, or learned from available incomplete data, based on which an optimal control is designed and then applied to the actual system. This gives rise to the problem of performance loss due to the mismatch between the actual system and the assumed system. A robustness problem in this context is to show that the error due to mismatch decreases to zero as the assumed system approaches the actual system. 

This is a problem of major practical importance. Most of the existing works in this direction are concerned with discrete-time Markov decision processes, see for instance \cite{KY-20}, \cite{KRY-20}, \cite{BJP02}, \cite{KV16}, \cite{NG05} \cite{SX15}, and references therein. For such problems in discrete-time, as noted in several explicit examples in \cite[Section 1.3.2]{KY-20}, the modeling and approximation errors can be in the functional representation of controlled transition dynamics, and/or in the distribution of the driving noise process. 

On the other hand, the literature on robustness of stochastic optimal control for continuous-time systems seems to be somewhat limited in their scope, as we detail in the following; for a collection of studies in continuous-time we refer the reader to e.g. \cite{GL99}, \cite{LJE15} \cite{pradhan2022robustness}, \cite{PDPB02a}, \cite{PDPB02b}, \cite{PDPB02c}, \cite{RZ21}. 

Notably, if one considers the following nominal model:
\[d X(t) \,=\, b(X(t),U(t)) dt + \sigma(X(t))dW(t)\, \]
with $X_0=x\in \mathbb{R}$, the uncertainty may be in the drift term $b$, in the diffusion term $\sigma$, or in the noise itself.

This distinction (on the source of modeling errors) turns out to be crucial when one studies robustness in continuous-time. In a recent work, \cite{pradhan2022robustness} studied the robustness problem when the model is perturbed via the changes in both $b$ and $\sigma$. The work established continuity of value functions (optimal cost as a function of initial state) and provided sufficient conditions which ensure robustness/stability (of optimal controls designed under model uncertainties when applied to a true system), under discounted, ergodic and cost-up-to-an-exit-time criteria. The paper utilized regularity properties of solutions to the optimality (HJB) equations. {\it Continuous convergence in control actions} of models and cost functions was shown to be a unifying condition for continuity and robustness properties in both the discrete-time setup, studied in \cite{KY-20} (discounted cost) and \cite{KRY-20} (average cost), and continuous-time. Also, it is important to note that under appropriate smoothness assumptions on $\sigma$, the robustness of optimal control for Stratonovich model (\ref{EAM}) under parametric perturbation of the model follows from \cite{pradhan2022robustness}\,. 

However, for many practical systems, it is not the drift and diffusion terms, but the actual driving noise itself which is incorrectly modelled. In particular, despite the tremendous success of stochastic analysis building on the It\^o theory involving Brownian noise and associated semi-martingale diffusions, due to the structural and regularity properties of Brownian noise facilitating a versatile mathematical analysis, it should be emphasized that the Brownian noise is only an idealization and physical models are typically never precisely Brownian leading to a ubiquitious robustness question on the nature of the driving noise process. 

For this problem, the current machinery available in stochastic control theory does not provide a complete theory for such a practically important problem. There are several reasons for this lack of a complete theory of robustness: 

\begin{itemize}
\item \textbf{Inapplicability of It\^o Theory.} One of the issues with It\^o theory is that in general, the solution map taking a sample path of the driving signal to the solution lacks continuity in the driving signal. Consequently, It\^o theory is poorly suited to study variations in the noise. In addition, It\^o calculus only works for semi-martingales. So non-semimartingales such as the fractional Brownian motion which allows for correlations in the increments, is prohibited. See \sy{\cite{Protter-Book,ikeda2014stochastic}} for an introduction to It\^o calculus with semi-martingales.

\item \textbf{Lack of Absolute Continuity.} We should also add that risk-sensitive control methods \cite{dupuis2000robust,albertini2001small,dupuis2000kernel} allow one to obtain performance bounds of control policies designed for a nominal model applied to a true model. Risk-sensitive cost methods have powerful design implications and connections with a game theoretic formulation \cite{jacobson1973optimal,basbern}. However, when the noise processes are Brownian and do not have identical diffusion terms $\sigma$, the mathematical framework cannot support the risk-sensitive formulation as relative entropy {\it apriori} requires absolute continuity of path measures. Accordingly, this approach is not applicable when we deviate from the Brownian paradigm even slightly. 

\item \textbf{Partial Results.} Nonetheless, this problem had received significant attention until a few decades ago: For ``wide bandwidth noise" driven controlled system \cite{K90}, \cite{KR87}, \cite{KR87a}, \cite{KR88}, a diffusion approximation technique was used to study the stochastic optimal control problems. However, the assumptions on the ``wide bandwidth" noise are complicated and difficult to interpret. We note that this body of work pioneered by Kushner and colleagues strictly build on a martingale representation requirement and essentially imposes that the noise distribution satisfy a martingale characterization whose limit is the Brownian, building also on \cite{kurtz1975semigroups}; this limits the class of noise processes which approximates the Brownian noise.

There have also been so called Wong-Zakai theorems \cite{wong1965convergence} for uncontrolled stochastic differential equations (see \cite{Wong-Zakai-Survey} for a survey). Loosely, these state if smooth $W^n$ converge to the Brownian motion $W$ almost surely (in some appropriate sense) then the solutions of the ODEs converge to that of the Stratonovich limit almost surely. This can be seen as an early form of an ingredient needed for robustness. However, these involve only uncontrolled SDEs. For controlled SDEs, both the drift function and the noise might vary. 

\end{itemize}

In view of the above, the problem of robustness to disturbance models is an important, yet unfinished problem. 

In this paper, we will present a general solution to noise approximation problem via the relatively recent theory of rough paths (see \cite{Friz-Victoir-Book,Friz-Hairer-Book}), when we can show that the approximating noise process approaches the Brownian limit in an appropriate sense. This approximation will in particular include wide-band noise approximations (whose limit in the Fourier domain would be the white noise, as commonly called in engineering systems).

Rough paths theory was introduced in \cite{lyons1998differential}. The key point of rough paths theory is that enhancing a signal with its iterated integrals gives more information that sidesteps several of the issues of It\^o theory. Several of its achievements are the following.

\begin{itemize}
    \item \textbf{General Noise} Rough paths theory allows for any noise that is either H\"older continuous or has finite $p$-variation.
    \item \textbf{Continuity of solution map over noise} The It\^o solution map that takes a sample path of a Brownian motion (or general semi-martingale) to the solution of the SDE is in general discontinuous. That is, Brownian paths that are close can have wildly different behavior. However by changing the topology on the space of paths to include the iterated integrals, continuity is restored.
    \item \textbf{Joint continuity over coefficients and noise} The standard Wong-Zakai theorem is a kind of continuity result of Stratonovich differential equations. However, the standard theory does not allow the coefficients and the noise to simultaneously change. Rough paths theory lets both the coefficients and the noise to change at the same time.
    \item \textbf{Pathwise solutions} Solutions to It\^o SDEs are defined in an $L^2$ or limit in probability sense. Rough paths theory allows for SDEs to be defined for each sample path.
\end{itemize}
Rough paths theory is one of the crucial ingredients we use to prove robustness. Notably both It\^o and Stratonovich equations can be understood in the rough paths sense, as the latter generalizes the former integration theories. There is in particular a close relationship between rough paths and Stratonovich integration; e.g \cite{Rough-Trap} shows that the Stratonovich integral of general Gaussian processes coincides with a ``rough integral". 

During our literature research, we also realized that much of the questions raised in this paper received significant attention in the 1960s, and the research motivated what we may refer to as early versions of a rough path theory, which since seem to have been overlooked. In particular, we highlight the connection with the belated integral of McShane (see \cite{McShane-Belated}), which provides a stochastic integration with the desirable property that the solution is continuous in the driving noise and solutions can be defined pathwise. Notably, McShane does this by including not only the process but also the iterated integrals above the process. This is reminiscent of rough paths theory, although the assumptions are quite restrictive. In fact, the assumptions on allowable processes force the process to be a semi-martingale, so it does not extend the class of allowable signals. However the belated integral is defined pathwise and does have continuity. We refer the reader to \cite{kushner2014partial} for an excellent historical perspective.

Stochastic control via rough paths is an emerging subject. One aspect on the pathwise nature of solutions to rough differential equations (RDEs) is that the problems are no longer Markovian or sometimes non-causal. In one line of research, to avoid the anticipative nature of non-Brownian noise and control policies, a Lagrangian penalty approach has been proposed in \cite{DFG-Control}, building on Rogers (see \cite{Rogers-Pathwise}) in discrete time; the approach's primary utility is conceptual as the penalty method imposing non-anticipativity essentially requires the optimal cost to be known. There have also been several papers on non-Markovian models including those that have studied optimal stochastic control with fractional Brownian motion, such as \cite{fBm-control-3,fBm-control-1,fBm-control-2}, for the case of either linear models (where completion of squares methods are shown to be applicable) or more general models where  correlations between the increments are positive. We note that rough paths theory allows for both positive or negative correlations. As an application to financial mathematics, \cite{bauerle2020portfolio} considered systems modeled with fractional Brownian noise and non-Markovian dynamics with dynamics depending on the history of the process.

On filtering theory, \cite{10.1214/19-AAP1558} applies rough path theory to optimal filtering to establish robustness to modeling errors, which also serves as a counterpart to \cite{10.1214/12-AAP896} which has considered the continuity of filters in measurement realizations for a general class of multi-dimensional signal-measurement correlated diffusion models (with the measurement acting as a driving noise, where rough path theory is to be utilized unlike the classical robustness theory involving regular diffusions; see e.g. \cite{clark1978design,clark2005robust}).

{\bf Contributions.} In our paper, as our main result on robustness to Brownian idealizations of the driving noise to be presented in Section \ref{MainResSec}, we show that within the class of Lipschitz continuous control policies, an optimal solution for the Brownian idealized model is near optimal for the true system driven by a non-Brownian (but near-Brownian) noise. We will also state generalizations under relaxed control policies. We will consider the discounted cost criterion and a finite horizon cost criterion. See Section \ref{mainResS} for an explicit statement and Section \ref{MainResSec} for a complete analysis. 

As some examples of near-Brownian but non-Brownian motion we have the following, which will be defined explicitly later in the paper.
\begin{itemize}
    \item The \textbf{Wong-Zakai approximation}, which is an approximation of the Brownian motion that is piecewise linear.
    \item The \textbf{Karhunen-Lo\`eve approximation}, which is an approximation through what are essentially random Fourier series.
    \item The \textbf{mollified Brownian motion}, which is the Brownian motion convolved with a smooth mollifier.
    \item The \textbf{fractional Brownian motion}, which is a generalization of Brownian motion to allow for correlations in the increments.
\end{itemize}
The above is a non-exhaustive list and all we need is almost sure convergence of the process along with its iterated integrals in a suitable topology (the \textit{rough topology}). The last example can be seen as one that is inherent to rough paths. Fractional Brownian motion allows for correlations in the noise while still being Gaussian and enjoying the same properties of self similarity and stationary increments that Brownian motion does. However fractional Brownian motion is not a semi-martingale unless it is the Brownian motion, necessitating the use of tools other than It\^o calculus. 

Additionally, sometimes the integration theory itself is incorrectly modelled. For an example of this phenomenon see \cite{Escudero-1,Escudero-2} which considers an insider trading model under a variety of anticipative stochastic integrals. The authors show that under some choices the insider trader makes less money than the fair trader, demonstrating the need of a critical analysis of integration theories. A priori there is no reason to use one theory of integration over another and the theory of integration itself can be seen as a choice of model. In this context, rough paths theory can be seen as a way of parameterizing integration theories. In addition to showing robustness with respect to the noise, our analysis also leads to implications on robustness with respect to the interpretations of integration theory itself.

\section{Description of the Problem}

\subsection{True Model and Approximate Models}

\textbf{True Model:}
\begin{equation}\label{ETM}
d X^{\epsilon}(t) = b(X^{\epsilon}(t), U(t))dt + \sigma(X^{\epsilon}(t))d \xi^{\epsilon}(t)  
\end{equation} where $\xi^{\epsilon}(t)$ is a near-Brownian noise process\,, {$U(\cdot)$ is a $\cP(\Act \times [0, \infty))$ valued process such that $U(\Act \times [0,t]) = t$ for all $t \geq 0$\,, where $\Act$ is the action space which is assumed to be a compact metric space. The process $U(\cdot)$ is said to be an {\it admissible relaxed} control if $\int_0^{t}\int_{\Act} f(s, \zeta) U(\D s, \D \zeta)$ is progressively measurable with respect to the $\sigma$-algebras $\sF_t^{\epsilon} \df \sigma[0, t]\times \sigma\{\xi^{\epsilon}(s), s\leq t\} $ for each bounded continuous function $f(\cdot)$\,. Let $\Uadm_{T^{\epsilon}}$ be the space of all admissible controls}\,.\\

Here, it is not clear {\it apriori} what one means by the stochastic integration term $\sigma(X(t))d \xi^{\epsilon}(t)$. This will be discussed later in the paper.

The approach is to solve an idealized problem, obtain the optimal solution and apply it to the true model. The goal is to show robustness (or stability) of solutions as the approximating model and the true model are sufficiently close.
\[\quad\]

\noindent\textbf{Approximating Idealized Model:}
\begin{equation}\label{EAM}
d X(t) = b(X(t), U(t))dt + \sigma(X(t))\circ d W(t)    
\end{equation} where $W$ is a $d$-dimensional Wiener process defined on a complete probability space $(\Omega, \mathcal{F}, \mathrm{P})$\,, {$U(\cdot)$ is a $\cP(\Act \times [0, \infty))$ valued process such that for each bounded continuous function $f(\cdot)$ and $t>0$, $\int_0^{t}\int_{\Act} f(s, \zeta) U(\D s, \D \zeta)$ is independent of $W(t+s) - W(t)$ for any $s>0$. Such a process $U(\cdot)$ is called an {\it admissible relaxed} control for the idealized model. The space of all admissible controls for the idealized model is denoted by $\Uadm_I$\,.}

In view of \cite[Theorem~1.2, 1,.4]{Ikeda-Watanabe-book}, we have that 
\begin{align*}
\sigma(X(t))\circ d W(t) &=  \sigma(X(t)) d W(t) + \frac{1}{2}d \sigma(X(t)) d W(t) \\
& = \sigma(X(t)) d W(t) + \frac{1}{2} \frac{\partial\sigma(X(t))}{\partial x}dX(t) d W(t) \\
& = \sigma(X(t)) d W(t) + \frac{1}{2} \frac{\partial\sigma(X(t))}{\partial x}\sigma(X(t))dt\,.
\end{align*} Thus the approximating models (\ref{EAM}) are equivalent to the following It$\hat{\rm o}$ diffusion models
\begin{equation}\label{EAMEqui}
d X(t) = \hat{b}(X(t), U(t))dt + \sigma(X(t)) d W(t) \,, \end{equation} where
\begin{equation*}
\hat{b}(x, u) = b(x, u) + \frac{1}{2}\frac{\partial\sigma(x)}{\partial x}\sigma(x) 
\end{equation*}

In order to ensure the existence and uniqueness of strong solutions of (\ref{EAMEqui}), we impose the following assumptions on the drift $b$ and the diffusion matrix $\sigma$\,. 
\begin{itemize}
\item[\hypertarget{A1}{{(A1)}}]
\emph{Lipschitz continuity:\/}
The function
$\sigma\,=\,\bigl[\sigma^{ij}\bigr]\colon\mathbb{R}^{d}\to\mathbb{R}^{d\times d}$ is in $C_b^3(\mathbb{R}^d; \mathbb{R}^{d\times d})$, and $b\colon \mathbb{R}^d\times\Act\to\mathbb{R}^d$ is Lipschitz continuous in $x$ and $\zeta$. In particular, for some constant $C>0$, we have
\begin{equation}\label{ELipForExis}
\abs{b(x,\zeta_1) - b(y, \zeta_2)} \,\le\, C\,\left(\abs{x-y} + \abs{\zeta_1 - \zeta_2}\right)
\end{equation}
for all $x,y\in \Rd$ and $\zeta_1, \zeta_2\in\Act$\,. Also, we are assuming that $b$ is jointly continuous in $(x,\zeta)$.

\medskip
\item[\hypertarget{A2}{{(A2)}}]
\emph{Boundedness condition:\/}
The function $b$ is uniformly bounded, i.e., for some positive constant $C_0$, we have
\begin{equation*}
\sup_{\zeta\in\Act}\, \abs{b(x, \zeta)}  \,\le\,C_0 \qquad \forall\, x\in\RR^{d}\,.
\end{equation*}

\medskip
\item[\hypertarget{A3}{{(A3)}}]
\emph{Nondegeneracy:\/}
For each $R>0$, it holds that
\begin{equation*}
\sum_{i,j=1}^{d} a^{ij}(x)z_{i}z_{j}
\,\ge\,C^{-1} \abs{z}^{2} \qquad\forall\, x\in \Rd\,,
\end{equation*}
and for all $z=(z_{1},\dotsc,z_{d})\transp\in\RR^{d}$,
where $a\df \frac{1}{2}\sigma \sigma\transp$.

In view of Assumption (\hyperlink{A1}{{(A1)}}), it is easy to see that $\hat{b}$ also satisfies (\ref{ELipForExis})\,.   
It is well known that under the hypotheses \hyperlink{A1}{{(A1)}}--\hyperlink{A3}{{(A3)}}, for each $U\in \Uadm_{I}$ there exists a unique strong solution of (\ref{EAMEqui}) (see \cite[Theorem~2.2.4]{ABG-book}), and for any Markov strategy (i.e., it depends only on the current time and state of the system), (\ref{EAMEqui}) admits a unique strong solution which is also a strong Feller process (see \cite[Theorem~2.2.12]{ABG-book} and the discussion therein for a historical overview)\,.

Let $c\colon\Rd\times\Act \to \RR_+$ be the \emph{running cost} function. We assume that 
\noindent
\begin{itemize}
\item[\hypertarget{A4}{{(A4)}}]
The \emph{running cost} $c$ is bounded (i.e., $\|c\|_{\infty} \leq M$ for some positive constant $M$), jointly continuous in $(x, \zeta)$ and locally Lipschitz continuous in its first argument uniformly with respect to $\zeta\in\Act$.
\end{itemize}
\end{itemize}
\subsection{Cost Evaluation Criteria}
In this manuscript, we are interested in the problem of minimizing infinite horizon discounted cost and finite horizon cost, respectively:

\subsection*{Discounted cost criterion} 
For $U \in\Uadm_{I}$, the associated \emph{$\alpha$-discounted cost} is given by
\begin{equation}\label{EDiscost}
\cJ_{\alpha, I}^{U}(x, c) \,\df\, \Exp_x^{U} \left[\int_0^{\infty} e^{-\alpha t} c(X(s), U(s)) \D s\right],\quad x\in\Rd\,,
\end{equation} where $\alpha > 0$ is the discount factor
and $X(\cdot)$ is the solution of (\ref{EAM}) corresponding to $U\in\Uadm_{I}$ and $\Exp_x^{U}$ is the expectation with respect to the law of the process $X(\cdot)$ with initial condition $x$. The controller tries to minimize (\ref{EDiscost}) over his/her admissible policies $\Uadm_{I}$\,. Thus, a policy $U^{*}\in \Uadm_{I}$ is said to be optimal if for all $x\in \Rd$ 
\begin{equation}\label{OPDcost}
\cJ_{\alpha}^{U^*}(x, c) = \inf_{U\in \Uadm_{I}}\cJ_{\alpha}^{U}(x, c) \,\,\, (\,=:\, \,\, V_{\alpha}^{I}(x))\,,
\end{equation} where $V_{\alpha}(x)$ is called the optimal value.

For the true models, for $U\in\Uadm_{T^{\epsilon}}$ the associated discounted cost is defined as
\begin{equation}\label{EApproDiscost}
\cJ_{\alpha, T^{\epsilon}}^{U}(x, c) \,\df\, \Exp_x^{U} \left[\int_0^{\infty} e^{-\alpha t} c(X^{\epsilon}(s), U(s)) \D s\right],\quad x\in\Rd\,,
\end{equation} where $X^{\epsilon}$ is the solution of (\ref{ETM}) under $U\in\Uadm_{T^{\epsilon}}$\,. The optimal value is defined as 
\begin{equation}\label{EApproOptDisc}
V_{\alpha}^{T^{\epsilon}}(x) \,\df\, \inf_{U\in\Uadm_{T^{\epsilon}}}\cJ_{\alpha, T^{\epsilon}}^{U}(x, c)
\end{equation}
\subsection*{Finite horizon cost criterion}
Let $T>0$ be the fixed time horizon. For each $U\in\Uadm_{I}$ the associated finite horizon cost is defined as
\begin{equation}\label{EFCost1}
\cJ_{T, I}^{U}(x, c) \,\df\, \Exp_x^{U} \left[\int_0^{T} c(X(s), U(s)) \D s + H(X(T))\right],\quad x\in\Rd\,,
\end{equation} where $X(\cdot)$ is the solution of (\ref{EAM}) corresponding to $U\in\Uadm_{I}$, and $H$ is the terminal cost\,. We assume that $H\in \cC_{b}(\Rd)$\,. The optimal value is defined as
\begin{equation}\label{ECost1Opt}
\cJ_{T, I}^{*}(x, c) \,\df\, \inf_{U\in \Uadm_{I}}\cJ_{T, I}^{U}(x, c)\,.
\end{equation}
Then a control $U^*\in \Uadm_{I}$ is said to be optimal if we have 
\begin{equation}\label{ECost1Opt1}
\cJ_{T, I}^{U^*}(x, c) = \cJ_{T, I}^{*}(x, c)\,, \quad x\in\Rd\,.
\end{equation}
Similarly, for each $U\in\Uadm_{T^{\epsilon}}$ the associated finite horizon cost in true model is defined as   
\begin{equation}\label{EFCostAprox1}
\cJ_{T, T^{\epsilon}}^{U}(x, c) = \Exp_x^{U}\left[\int_0^{T} c(X^{\epsilon}(s), U(s)) \D{s} + H(X^{\epsilon}(T))\right]\,,
\end{equation} where $X^{\epsilon}$ is the solution of (\ref{ETM}) under $U\in\Uadm_{T^{\epsilon}}$\,. The optimal value is defined as
\begin{equation}\label{EFCost1OptAprox}
\cJ_{T, T^{\epsilon}}^{*}(x, c) \,\df\, \inf_{x\in\Rd}\inf_{U\in \Uadm_{T^{\epsilon}}}\cJ_{T, T^{\epsilon}}^{U}(x, c)\,.
\end{equation}

\subsection{Existence of Optimal Policies of the Approximating Models}
Let us deﬁne a parametric family of operators $\sL_{\zeta}$ mapping $\cC^{2}(\Rd)$ to $\cC(\Rd)$ by
\begin{equation*}
{\sL}_{\zeta} f(x) \df\, \trace\bigl(a(x)\grad^2 f(x)\bigr) + \,\hat{b}(x,\zeta)\cdot \grad f(x) \,.  
\end{equation*} for all $\zeta\in \Act$ and $f\in \cC^{2}(\Rd)$\,. Now, from \cite[Theorem~3.5.6]{ABG-book}, we have the following complete characterization of the discounted optimal controls for each approximating models.
\begin{theorem}\label{TD1.1}
Suppose Assumptions (A1)-(A4) hold. The optimal discounted cost $V_{\alpha}^I$ defined in (\ref{EApproOptDisc}) is the unique solution in $\cC^2(\Rd)\cap\cC_b(\Rd)$ of the HJB equation
\begin{equation}\label{AppOptDHJB}
\min_{\zeta \in\Act}\left[\sL_{\zeta}V_{\alpha}^I(x) + c(x, \zeta)\right] = \alpha V_{\alpha}^I(x) \,,\quad \text{for all\ }\,\, x\in\Rd\,.
\end{equation}
Moreover, $v^*\in \Usm$ is $\alpha$-discounted optimal control if and only if it is a measurable minimizing selector of (\ref{AppOptDHJB}), i.e.,
\begin{equation}\label{AppOtpHJBSelc}
\hat{b}(x,v^*(x))\cdot \grad V_{\alpha}^I(x) + c(x, v^*(x)) = \min_{\zeta\in \Act}\left[ \hat{b}(x, \zeta)\cdot \grad V_{\alpha}^I(x) + c(x, \zeta)\right]\quad \text{a.e.}\,\,\, x\in\Rd\,.
\end{equation}
\end{theorem}

Theorem~\ref{TD1.1}, implies existence of a deterministic optimal policy $v^*(\cdot)\in \Usm$\,. Thus by Lusin's theorem and Tietze's extension theorem as in \cite[Theorem~4.2]{pradhan2022near}, for any $\delta > 0$ we have a continuous function $\Tilde{v}_{\delta}^*$ which is $\delta$ optimal\,. Let $\phi \in \cC_c^{\infty}(\Rd)$ with $\phi = 0$ for $|x|>1$ and $\int_{\Rd}\phi(x) d x = 1$\,. Now for $\epsilon>0$ define $\Tilde{v}_{\delta, \epsilon}^*(x) := \epsilon^{-d}\phi(\frac{x}{\epsilon})* \Tilde{v}_{\delta}^*(x)$ (convolution with respect to $x$). Since, $\Tilde{v}_{\delta}^*$ is continuous it is well known that $\Tilde{v}_{\delta, \epsilon}^*(x) \to \Tilde{v}_{\delta}^*(x)$ for a.e $x\in \Rd$\,. Thus, by dominated convergence theorem we deduce that
\begin{equation}\label{ENearOptiA}
\int_{\Rd} f(x)g(x, \Tilde{v}_{\delta, \epsilon}^*(x)) dx \to \int_{\Rd} f(x)g(x, \Tilde{v}_{\delta}^*(x)) dx      
\end{equation} for all $f\in L^1(\Rd)\cap L^{2}(\Rd)$ and $g\in \cC_{b}(\Rd\times \Act)$\,.
Also, for each $\epsilon > 0$, it is easy to show that $\frac{\partial \Tilde{v}_{\delta, \epsilon}^*}{\partial x^i}$ is bounded for $i= 1,\dots , d$\,. This, in particular implies that for each $\epsilon >0$, $\Tilde{v}_{\delta, \epsilon}^*$ is uniformly Lipschitz continuous in  $\Rd$\,. Therefore, in view of \cite[Theorem~3.1]{pradhan2022near} and (\ref{ENearOptiA}), it follows that uniformly Lipschitz continuous functions are near optimal for the discounted cost criterion\,.

For the finite horizon case, from \cite[Theorem~3.3, p. 235]{BL84-book}, we have the existence of a unique solution to the finite horizon optimality equation (\ref{EFinitecost1A}) (or, the HJB equation). Applying It\^{o}-Krylov formula (as in \cite[Theorem~3.5.2]{HP09-book}), we have existence of a finite horizon optimal Markov policy (for more details see \cite[Section~5]{pradhan2022robustness})\,. In particular we have the following theorem\,.
\begin{theorem}
Suppose Assumptions (A1)-(A4) hold. Then 
\begin{align}\label{EFinitecost1A}
&\frac{\partial \psi}{\partial t} + \inf_{\zeta\in \Act}\left[\sL_{\zeta}\psi + c(x, \zeta) \right] = 0 \nonumber\\
& \psi(T,x) = H(x)
\end{align} admits a unique solution $\psi\in \Sob^{1,2,p,\mu}((0, T)\times\Rd)\cap \Lp^{\infty}((0, T)\times\Rd)$, for some $p\ge 2$ and $\mu > 0$. Moreover, there exists $v^*\in \Um$ such that $\cJ_{T}^{v^*}(x, c) = \cJ_{T}^*(x, c) = \psi(0,x)$\,.
\end{theorem}
In view of the continuity result as in \cite[Theorem~6.1]{pradhan2022near}, and by similar convolution argument (as in the discounted cost case), one can prove the near optimality of the uniformly Lipschitz continuous functions for the finite horizon cost criterion\,.

\subsection{Problem Statement}\label{mainResS}

For the criteria noted above, our goal is to show that
\begin{itemize}
\item \textbf{For Discounted Cost:} For any Lipschitz continuous $\delta$- optimal control $v^{\delta}$ of the idealized model, we have
\begin{equation}\label{EdisNearOpt} 
\cJ_{\alpha, T^{\epsilon}}^{v_{\delta}}(x, c) \leq V_{\alpha}^{T^{\epsilon}} + 2\delta\,. 
\end{equation} for small enough $\epsilon$\,.
\item \textbf{For Finite Horizon Cost:} For any Lipschitz continuous $\delta$- optimal control $v^{\delta}$ of the idealized model, we have
\begin{equation}\label{EFiniteHNearOpt} 
\cJ_{T, T^{\epsilon}}^{v^{\delta}}(x, c) \leq \cJ_{T, T^{\epsilon}}^{*} + 2\delta\,,
\end{equation} for small enough $\epsilon$\,. 
\end{itemize}

In words, a near-optimal policy for the idealized model is also near-optimal for the true model as the driving noise of the true model approaches the Brownian. 

\section{Supporting Results and Rough Path Topology}

In our paper, we will consider noise processes which will approach the Brownian in an appropriate topology, which is called the rough paths topology.

\subsection{Background: On Rough Paths Theory}

In this section we review essential elements of rough paths theory and rough differential equations. For a comprehensive introduction, see \cite{Friz-Hairer-Book,Friz-Victoir-Book}. 

Rough paths theory allows for a pathwise theory of solutions to differential equations driven by highly irregular signals. The fundamental observation of rough paths theory is that the issue of defining solutions to differential equations driven by an irregular signal $X=(X^1,...,X^d)$ can be reduced to defining the iterated integrals $\int_s^t (X^i(r)-X^i(s))dX^j(r)$. More precisely, rough paths theory treats H\"older continuous driving signals - or in the case of stochastic differential equations, stochastic processes that are almost surely H\"older. We recall the definition of H\"older continuity below.

\begin{definition}
    Define for $\alpha\in (0,1]$ the space $C^\alpha([0,T],\mathbb R^d)$ of $\alpha$-H\"older functions $f:[0,T]\to \mathbb R^d$ equipped with the norm $\|f\|_\alpha:=\sup_{t\neq s}\frac{|f(t)-f(s)|}{|t-s|^\alpha}.$
\end{definition}
If $X$ is a signal that is $\alpha$-H\"older continuous with $\alpha\in (1/3,1/2]$ and $F$ is a smooth function, then for a partition of $[0,t]$, $\mathcal P=\{0=t_0<...<t_n=t\}$ we have that (the a-priori ill defined) integral
\begin{align}
    \int_0^t F(X(r)) dX(r)&=\sum_{k=0}^n \int_{t_k}^{t_{k+1}} F(X(r))dX(r) \nonumber \\
    &=\sum_{k=0}^n \int_{t_k}^{t_{k+1}} F(X(t_k))+F'(X(t_k))(X(r)-X(t_k))+O(|r-t_k|^{2\alpha})dX(r) \nonumber \\
    &=\sum_{k=0}^n \bigg(F(X(t_k))(X(t_{k+1})-X(t_k))+F'(X(t_k))\int_{t_k}^{t_{k+1}}(X(r)-X(t_k)) dX(r) \nonumber \\
    & \quad \quad +O(|t_{k+1}-t_k|^{3\alpha})\bigg). \label{roughHeuristicIter}
\end{align}
As $3\alpha>1$, the remainder term should go to $0$. This reduces the problem of defining the integral $\int_0^t F(X(r))dX(r)$ to just defining $\int_{t_k}^{t_{k+1}}(X(r)-X(t_k)) dX(r)$. We may take the right hand side as a \textit{definition} of the left hand side, so long as we define the iterated integral first. However, if $X$ is irregular then the iterated integral does not exist ''canonically" as a Riemann-Stieltjes integral and therefore must be postulated.

A rough path above a signal $X$ is therefore a pair $\mathbf{X}_{s,t}=(X_{s,t}, \mathbb X_{s,t})$ where $X_{s,t}$ is the increment of $X$ and $\mathbb X_{s,t}$ is a \textit{definition} or \textit{postulation} of the iterated integral $\int_{s}^{t}(X(r)-X(s)) dX(r)$. We give a precise definition below.

\begin{definition}\label{def:rough-path}
    Given a signal $X\in C^\alpha([0,T], \mathbb R^d)$ with $\alpha\in (1/3,1/2]$ we say $\mathbf X=(X, \mathbb X):\Delta_T^2\to \mathbb R^d\oplus \mathbb R^{d\times d}$ is a \textbf{rough path} above $X$ if for all $s,u,t\in [0,T]$ we have
    \begin{align*}
    &(i) \qquad X_{s,t}=X(t)-X(s)\\
        &(ii) \qquad\mathbb X_{s,t}-\mathbb X_{s,u}-\mathbb X_{u,t}=X_{s,u}\otimes X_{u,t}\\
        &(iii)\qquad \|\mathbf X\|_{\alpha, 2\alpha}:=\sup_{t\neq s}\frac{|X_{s,t}|}{|t-s|^\alpha}+\sup_{t\neq s}\frac{|\mathbb X_{s,t}|}{|t-s|^{2\alpha}}<+\infty.
    \end{align*}
Here, $\Delta_T^2=\{(s,t):0\leq s\leq t\leq T\}$ denotes the $2$-simplex. We denote the set of rough paths by $\mathcal C^{\alpha}$ and the set of rough paths above $X$ by $\mathcal C_X^{\alpha}$. The topology generated by the seminorm $\|\cdot\|_{\alpha, 2\alpha}$ is called the \textbf{rough topology}.
\end{definition}

One can check that if $X\in C^1$ and $\int_s^t (X(r)-X(s)) dX(r)$ is the Riemann-Stieltjes integral, $\mathbf{X}_{s,t}:=(X(t)-X(s),\int_{s}^{t}(X(r)-X(s)) dX(r))$ satisfies Definition \ref{def:rough-path}. In this sense, there is a natural embedding of $C^1$ smooth functions into the space of rough paths. The closure of these functions under the rough topology is the space of geometric rough paths.

\begin{definition}
    For $\alpha\in (1/3,1/2]$, denote by $\mathring{\mathcal C}_g^{\alpha}$ the image of the embedding $\iota: C^1([0,T],\mathbb R^d)\to\mathcal C^\alpha$ where $\iota(f)_{s,t}=(f(t)-f(s),\int_s^t (f(r)-f(s))f'(r) dr)$. Let $\mathcal C_g^\alpha$ be the closure of $\mathring{\mathcal C}_g^{\alpha}$ under the seminorm $\|\cdot \|_{\alpha, 2\alpha}$ defined in Definition \ref{def:rough-path}. $\mathcal C_g^\alpha$ is called the set of $\textbf{geometric rough paths}$.
\end{definition}
We recall the following lemma, from \cite{Friz-Hairer-Book}, Section 2.2. 
\begin{lemma}
    Let $\mathbf X=(X, \mathbb X)\in \mathcal C_g^\alpha$ for some $\alpha \in (1/3,1/2]$. Then for all $(s,t)\in \Delta_T^2$ we have 
    \begin{equation}\label{eq:weak-geometric}
        \frac{1}{2}(\mathbb X_{s,t}+\mathbb X_{s,t}^t)=\frac{1}{2}X_{s,t}\otimes X_{s,t},
    \end{equation}
    where $\mathbb X^t$ is the transpose of $\mathbb X$.
\end{lemma}

\begin{remark}
    The set of rough paths satisfying equation \eqref{eq:weak-geometric} is often called the set of \textbf{weakly geometric} rough paths. See \cite{Friz-Hairer-Book}, Section 2.2 for a discussion on why the distinction between geometric and weakly geometric rarely matters. 
\end{remark}

Given a rough path, we may define the class of admissible integrands as the set of \textit{controlled rough paths}. 

\begin{definition}\label{def:controlled-rough-path}
    Given a function $X\in C^\alpha$ with $\alpha\in (1/3,1/2]$ we say that $y\in C^\alpha$ is \textbf{controlled by} $X$ if there exists a matrix $y'\in C^\alpha$ (called the \textbf{Gubinelli derivative}) so that
    \begin{equation}
        y(t)-y(s)=y'(s)(X(t)-X(s))+R_Y(s,t),
    \end{equation}
    where $R_Y(s,t)\in C^{2\alpha}$. The pair $(y,y')$ is a controlled rough path.
\end{definition}
\begin{remark}
The terminology in Definition \ref{def:controlled-rough-path} might be misleading. First - a controlled rough path is not a rough path. Second - a controlled rough path is not controlled in the sense of control theory. It is simply the class of integrands and includes solutions to differential equations driven by $X$, $F(X)$ for smooth $F$ and other types of integrands that have a ``Taylor-type" dependence on $X$. Also note that there is no reference to rough path in the definition of controlled path. Additionally, the Gubinelli derivative may not be unique.
\end{remark}

Continuing along the analysis in (\ref{roughHeuristicIter}), we present the following theorem on integration of rough paths and on the continuity and solution properties of RDEs.

\begin{theorem}[\cite{Friz-Hairer-Book}, Theorem 4.10]
    Let $T>0$. Let $\mathbf{X}=(X,\mathbb X)$ be a rough path as in Definition \ref{def:rough-path}. Let $(y,y')$ be controlled by $X$ as in Definition \ref{def:controlled-rough-path}. Then the following limit exists
    \begin{equation}\label{eq:rough-integral}
        \int_0^T y d\mathbf X=\lim_{|\mathcal P|\to 0}\sum_{[s,t]\in \mathcal P} y_sX_{s,t}+y_s' \mathbf X_{s,t}
    \end{equation}
    for any sequence of partitions with $\mathcal P\to 0$. The expression in \eqref{eq:rough-integral} is called the \textbf{rough integral} of $y$ against $\mathbf X$.
\end{theorem}
We cite the following theorem, which gives continuity not only in the driving signal but also in the coefficients. This is one of the main achievements of rough paths theory. 
\begin{theorem}[\cite{Friz-Victoir-Book}, Theorem(12.11)]\label{ThmCont1A}
   Let $\mathbf X^1$ and $\mathbf X^2$ be two rough paths as in Definition \ref{def:rough-path}. Consider the two rough differential equations
    \begin{align*}
        dY^1&=b_1(Y^1)dt+\sigma(Y^1) d\mathbf X^1\\
        dY^2&=b_2(Y^2)dt+\sigma(Y^2) d\mathbf X^2,
    \end{align*}
    where $\sigma\in C_b^3$ is thrice differentiable with bounded derivatives, $b_1$ and $b_2$ are Lipschitz, and $Y^1(0)=\xi_1$ and $Y^2(0)=\xi_2$. There there exist unique solutions $Y^1,Y^2$, that are controlled by $X$ in the sense of Definition \ref{def:controlled-rough-path}. Then we have that for all $\varepsilon>0$ there exists some $\delta>0$ so that if
    \begin{equation}
         |\xi_1-\xi_2|+\|\mathbf X^1-\mathbf X^2\|_{\alpha, 2\alpha}+\|b_1-b_2\|_{\infty}<\delta
    \end{equation}
    then 
    \begin{equation}
        \|Y^1-Y^2\|_\infty<\varepsilon.
    \end{equation}
\end{theorem}

\section{Noise Processes Approximating the Brownian in the Rough Paths Topology}
In this section, we give some examples of approximations that rough analysis can handle. 
\subsection{Finite Bandwith Karhunen-Lo\`eve Approximations}
\begin{proposition}[\cite{Karhunen-Loeve-Thesis}, Theorem 5.14]
    Let $W=(W_1,...,W_d)$ be a standard Brownian motion. Then for each $W^i$ we have the following Karhunen-Lo\`eve expansion
    \begin{equation}
        W_i(t)=\sqrt{2}\sum_{k=1}^\infty \frac{\sin((k-\frac12)\pi t)}{(k-\frac12)\pi} Z_k^i,
    \end{equation}
    where $Z_k^i\sim N(0,1)$ are i.i.d. and the convergence holds pointwise a.s. or a.s. in $L^2(0,T).$
\end{proposition}
We denote the partial sum by
\begin{equation}
    W_i^n(t)=\sqrt{2}\sum_{k=1}^n \frac{\sin((k-\frac12)\pi t)}{(k-\frac12)\pi} Z_k^i.
\end{equation}

The convergence can be improved to convergence under the rough topology, however, with the following proposition. See \cite{Friz-Victoir-Gaussian-1}, Theorem 35 and 37. Another reference is \cite{Jain-Monrad-Condition} Corollary 2.3 and Example 2.4 (with $H=1/2$).

In addition, the Karhunen-Lo\`eve expansion converges with probability one on the space of rough paths. Although we expect this to be known, we have not been able to find the precise statement and therefore for completeness we supply a short proof.
\begin{proposition}
    Let  $\mathbf W^n=(W^n,\mathbb W^n)$ be Karhunen-Lo\`eve Brownian motion enhanced with its canonical iterated integrals and let $\mathbf W:=(W,\mathbb W^{Strat})$ be Brownian motion enhanced with Stratonovich integrals. Then with probability one, $\mathbf W^n$ converges to $\mathbf W$ in the rough topology.  
\end{proposition}
\begin{proof}
First, we will check the covariance of $B^n - B$. To this end, see that for $1\leq i,j\leq d$ we have that the covariance is $0$ if $i\neq j$. If $i=j$ by Cauchy-Schwarz we have that
\begin{align*}
    |E[(W_i^n(t)-W_i(t))(W_j^n(s)-W_j(s))]|&\leq \sqrt{E[(W_i^n(t)-W_i(t))^2]E[(W_i^n(s)-W_i(s))^2]}.
\end{align*}
By Fatou-Lebesgue theorem we have
\begin{align*}
    E[(W_i^n(t)-W_i(t))^2]&= \frac{2}{\pi^2} E\left[\sum_{k,k'=n+1}^\infty (Z_k^i)^2 \frac{\sin((k-\frac12)\pi t)\sin((k'-\frac12)\pi t)}{(k-\frac12)(k'-\frac12)}\right]\\
&=\frac{2}{\pi^2}E\left[\lim_{N\to\infty}\sum_{k,k'=n+1}^N (Z_k^i)^2 \frac{\sin((k-\frac12)\pi t)\sin((k'-\frac12)\pi t)}{(k-\frac12)(k'-\frac12)}\right]\\
&\leq \frac{2}{\pi^2}\liminf_{N\to\infty}E\left[\sum_{k,k'=n+1}^N (Z_k^i)^2 \frac{\sin((k-\frac12)\pi t)\sin((k'-\frac12)\pi t)}{(k-\frac12)(k'-\frac12)}\right]\\
&\leq \frac{2}{\pi^2}\liminf_{N\to\infty}\sum_{k=n+1}^N  \frac{\sin^2((k-\frac12)\pi t)}{(k-\frac12)^2}\\
&\lesssim \frac{1}{n}.
\end{align*}
Therefore, using the main Theorem 1 of \cite{Friz-Victoir-Gaussian-Signals} we have that for all $\alpha\in (1/3,1/2)$ there is some positive $\theta>0$ so that for all $q\geq 1$ there is some $C_q>0$ so that 
\begin{equation}
    E[\rho_{\alpha}^q(\mathbf W^n,\mathbf W^{Strat})]\leq C_q n^{-q\theta}.
\end{equation}
By Borel-Cantelli, since for all $\varepsilon>0$ the series
\begin{align*}
    \sum_{n=1}^\infty P(\rho_{\alpha}(\mathbf W^n,\mathbf W^{Strat})>\varepsilon)&=\sum_{n=1}^\infty P(\rho_{\alpha}^q(\mathbf B^n,\mathbf W^{Strat})>\varepsilon^q)\\
    &\lesssim \sum_{n=1}^\infty n^{-q\theta}
\end{align*}
converges for $q$ large enough. We thus have the almost sure convergence on the space of rough paths. 
\end{proof}

\subsection{Wong-Zakai Approximations}
Wong-Zakai is another approximation we can use, that is a piecewise linear approximation. 
\begin{proposition}[\cite{Friz-Hairer-Book}, Exercise 10.2]
    Let $W^n$ be Wong-Zakai approximation of Brownian motion - that is it is piecewise linear with step size $1/n$, agreeing with Brownian motion on the grid points, and let $\mathbf W^n=(W^n,\mathbb W^n)$ be Wong-Zakai Brownian motion enhanced with its canonical Riemann-Stieltjes iterated integrals and let $\mathbf W:=(W,\mathbb W^{Strat})$ be Brownian motion enhanced with Stratonovich integrals. Then with probability one, $\mathbf W^n$ converges to $\mathbf W$ in the rough topology.  
\end{proposition}
\subsection{Mollified Brownian Motion Approximations}
Corollary 2.3 in \cite{Jain-Monrad-Condition} also implies that mollified Brownian motion converges in rough topology to Stratonovich Brownian motion in probability. By taking a subsequence, one can get almost sure convergence on the space of rough paths. That is, we have the following proposition. 
\begin{proposition}
    Let $W^n=W\ast \phi_n$ be a Brownian motion convolved with a smooth mollifier $\phi_n$ and denote by $(W^n, \mathbb W^n)$ the canonical rough path lifted by Riemann-Stieltjes integrals. Then (along a subsequence) we have that $(W^n, \mathbb W^n)$ converges in the rough path space $\mathscr C_g^\alpha$ for any $\alpha\in (1/3,1/2)$ with probability one to $(W,\mathbb W^{Strat})$
\end{proposition}

\subsection{Fractional Brownian Motion Approximations}
Brownian motion assumes independence of increments, which is a relatively strong assumption. Fractional Brownian motion (see \cite{fBm-book} for an introduction) is a generalization of Brownian motion that allows for correlations in the increments. More precisely, we have the following definition. 
\begin{definition}\label{def:fBm}
    For $H\in (0,1)$ (called the Hurst parameter), define a one dimensional $H$ fractional Brownian motion (fBm), $B_H$ as a centered Gaussian process so that
    \begin{enumerate}
        \item (Stationary increments) $B_H(t)-B_H(s)\sim B_H(t-s)$
        \item (Self similarity) $B_H(ct)\sim c^H B_H(t)$ for all $c\geq 0$.
    \end{enumerate}
    We also normalize so that $E[B_H^2(1)]=1$. 
\end{definition}
From Definition \ref{def:fBm} we can derive the covariance to be 
\begin{equation}
    E[B_H(t)B_H(s)]=\frac{1}{2}[t^{2H}+s^{2H}-|t-s|^{2H}].
\end{equation}
Notably, when $H=1/2$ the covariance reduces to $\min(s,t)$ which the covariance of a Brownian motion. Therefore we have that $\frac{1}{2}$-fBm is standard Brownian motion. Also, if $H>1/2$ the increments are positively correlated and if $H<1/2$ the increments are negatively correlated. Fractional Brownian motion can thus be seen as a generalization of Brownian motion to allow for correlations in the increments. Also, for all $\alpha<H$, fractional Brownian motion is almost surely $\alpha$-H\"older. Therefore making rough paths theory ideal to hande fBm in the case $H\in (1/3, 1/2)$.

Using Theorem 43 in \cite{Friz-Victoir-Gaussian-1} and the discussion on page 3, we can approximate weakly Stratonovich Brownian motion by fractional Brownian motion in the limit $H\to 1/2$. The spaces $\mathscr C_g^\alpha$ are separable (see section 2.2 in \cite{Friz-Hairer-Book}) and thus by Skorokhod's representation theorem there is some probability space so the convergence thus holds almost surely.

\subsection{Approximation of controlled Brownian SDEs}
The primary result from rough paths can be summarized in this concluding corollary of the above. 
\begin{corollary}\label{cor:approximation-Strat}
  Let $\mathbf W^n$ be rough path approximations to that converge to $\mathbf W^{Strat}$ which is Brownian motion enhanced with Stratonovich integration. Consider the two rough differential equations
    \begin{align*}
        dY^1&=b_1(Y^1)dt+\sigma(Y^1) d\mathbf W^n\\
        dY^2&=b_2(Y^2)dt+\sigma(Y^2) d\mathbf W^{Strat},
    \end{align*}
    where $\sigma\in C_b^3$ is thrice differentiable with bounded derivatives, $b_1$ and $b_2$ are Lipschitz, and $Y^1(0)=\xi_1$ and $Y^2(0)=\xi_2$. Then for almost all sample paths of Brownian motion, there exist unique solutions $Y^1,Y^2$, that are controlled by $W^n,W$ in the sense of Definition \ref{def:controlled-rough-path}. Additionally for almost all sample paths we have that for all $\varepsilon>0$ there exists some $\delta>0$ so that if
    \begin{equation}
         |\xi_1-\xi_2|+\|\mathbf W^n-\mathbf W^{Strat}\|_{\alpha, 2\alpha}+\|b_1-b_2\|_{\infty}<\delta
    \end{equation}
    then 
    \begin{equation}
        \|Y^1-Y^2\|_\infty<\varepsilon.
    \end{equation}
\end{corollary}
\begin{remark}
   The Stratonovich rough path and It\^o rough path are related by $\mathbf W ^{Strat}=(W,\mathbb W^{Strat})=(W,\mathbb W^{Ito}+\frac{1}{2}(t-s)I).$ Moreover, Theorem 9.1 in \cite{Friz-Hairer-Book} says that solutions to rough differential equations driven by It\^o or Stratonovich rough paths are in addition strong solutions to It\^o or Stratonovich stochastic differential equations. The same conversions from It\^o to Stratonovich stochastic differential equations hold.   
\end{remark}
\begin{remark}
    Corollary \ref{cor:approximation-Strat} should be compared with the classical Wong-Zakai result. See \cite{Wong-Zakai-Survey} Theorem 2.1 for a version of Wong-Zakai. The crucial difference is that Corollary \ref{cor:approximation-Strat} allows for changing drifts at the same time as changing driving noise. One other crucial difference is that the machinery of rough paths can handle general idealized noises such as fractional Brownian motion and other non semimartingales. Theorem \ref{ThmCont1A} is true for any rough path not just smooth approximations as in Wong-Zakai.  \cite{Jain-Monrad-Condition} shows that one can approximate general Gaussian noises in the rough path sense. 
\end{remark}
\section{Main Results}\label{MainResSec}

\subsection{General Robustness Theorem}
In the next section we present the robustness result for the infinite horizon discounted cost criterion\,.

\subsubsection{Analysis of the Discounted Cost}\label{AnalysisDis}
Since the noise term $\xi^{\epsilon}(\cdot)$ of the true model converges to the ideal Brownian noise $W(\cdot)$ in the rough path topology, in view of Theorem~\ref{ThmCont1A}, it follows that under suitable choice of control policies the true model (\ref{ETM}) converges to the idealized model (\ref{EAM}) (in sup norm) almost surely as $\epsilon \to 0$\,. 

For some positive constant $M_1$, let $$F_L \df \{h:\Rd \to U\mid |h(x) - h(y)| \leq M_1 |x - y|\,\,\, \text{for all}\,\,\, x,y \in \Rd \}\,.$$
Let the class of control policies induced by the functions in $F_L$ by $\Uadm_{LT^{\epsilon}} \df \{U_t\mid U_t = h(X^{\epsilon}(t)) \,\,\,\text{for some} \,\, h\in F_L\}$ for true model and similarly $\Uadm_{LI}$ for ideal model\,. By Arzel\'a-Ascoli theorem, we have that the set $\Uadm_{LT^{\epsilon}}$ is compact (under sup norm) for each $\epsilon >0$\,. If we consider the minimization problem (\ref{EApproOptDisc}) over the class of policies $\Uadm_{LT^{\epsilon}}$, the associated optimal value is defined by $V_{\alpha}^{LT^{\epsilon}} \,\df\, \inf_{U\in\Uadm_{LT^{\epsilon}}}\cJ_{\alpha, T^{\epsilon}}^{U}(x, c)$ (optimal discounted cost)\,.

Next proposition shows that for fixed Lipschitz continuous control policy, as $\epsilon\to 0$ (i.e., the true models approaches the idealized model) the associated cost in true model converges to the associated cost in the idealized model\,. 
\begin{proposition}\label{EforFixContConti1A}
Suppose Assumptions (A1)-(A4) hold. Then for any Lipschitz continuous policy $v\in \Usm$, we have 
\begin{align*}
 \cJ_{\alpha, T^{\epsilon}}^{v} \to  \cJ_{\alpha, I}^{v}\quad\text{as}\quad \epsilon\to 0\,.
\end{align*}
\end{proposition}
\begin{proof}
Since $\sigma\in \cC_b^3(\Rd)$ and for any Lipschitz continuous $v\in \Usm$, we have $b(\cdot, v(\cdot))$  is Lipschitz continuous, thus from Theorem~\ref{ThmCont1A} we have that the solution $X^{\epsilon}$ of (\ref{ETM}) converges to the solution $X$ of (\ref{EAM}) in sup norm as $\epsilon \to 0$\,. Thus by generalized dominated convergence theorem (see \cite[Theorem 3.5]{serfozo1982convergence}) we obtain our result\,.    
\end{proof}

Next, we prove certain continuity result which will be useful in proving our desired robustness result\,. 
\begin{proposition}\label{EForContValue1A}
Suppose Assumptions (A1)-(A4) hold\,. Then for any sequence of admissible policy $U^{\epsilon}\in \Uadm_{LT^{\epsilon}}$ with $U^{\epsilon}\to \hat{U}$, we have 
\begin{align*}
 \cJ_{\alpha, T^{\epsilon}}^{U^{\epsilon}} \to  \cJ_{\alpha, I}^{\hat{U}}\quad\text{as}\quad \epsilon\to 0,  
\end{align*} where $\hat{U}\in\Uadm_{LI}$\,.
\end{proposition}
\begin{proof} 
Since $U^{\epsilon}\in \Uadm_{LT^{\epsilon}}$, we have $U^{\epsilon}_t = h^{\epsilon}(X_t^{\epsilon})$ for some $h^{\epsilon}\in F_L$\,. By Arzel\'a-Ascoli theorem, we have $\hat{U}_t = \hat{h}(X_t)$ for some $\hat{h}\in F_L$\,. This implies that $\hat{U}\in\Uadm_{LI}$\,. Thus it is easy to see that $b(\cdot, h^{\epsilon}(\cdot))$ and $b(\cdot, \hat{h}(\cdot))$ are Lipschitz continuous functions. In view of Assumption (A1), this indeed implies that $b(\cdot, h^{\epsilon}(\cdot)) - b(\cdot, \hat{h}(\cdot))$ is uniformly Lipschitz continuous\,. Now, from Theorem~\ref{ThmCont1A}, we get that the solution $X^{\epsilon}$ of (\ref{ETM}) corresponding to $U^{\epsilon}$ converges to the solution $X$ of (\ref{EAM}) corresponding to $\hat{U}$ in $\|\cdot\|_\infty$ norm almost surely as $\epsilon \to 0$\,. Thus, by generalized dominated convergence theorem as in \cite[Theorem 3.5]{serfozo1982convergence}, we get our result\,.     
\end{proof}

Using Proposition~\ref{EForContValue1A}, we want to prove that in the limit the value functions are close to each other\,.
\begin{theorem}\label{EthmCont1A}
Suppose Assumptions (A1)-(A4) hold\,. If for any $\delta > 0$, $v_{\delta}^*$ is a $\delta$-optimal control of the idealized model, i.e., 
\begin{equation*}
   \cJ_{\alpha, I}^{v^{\delta}} \leq V_{\alpha}^{I} + \delta\,,
\end{equation*} with $|v_{\delta}^*(x) - v_{\delta}^*(y)| \leq \hat{M} |x - y|$ for some $0 < \hat{M} \leq M_1$\,.
Then, we have
\begin{equation*}
  \lim_{\epsilon \to 0} |V_{\alpha}^{LT^{\epsilon}} - V_{\alpha}^{I}| \leq \delta\,.  
\end{equation*}
\end{theorem}
\begin{proof}
In view of Theorem~\ref{ThmCont1A}, it is easy to see that (by the generalized dominated convergence theorem) for each $\epsilon > 0$, the map $U^{\epsilon}\to \cJ_{\alpha, T^{\epsilon}}^{U^{\epsilon}}$ is continuous on $\Uadm_{LT^{\epsilon}}$\,. Since $\Uadm_{LT^{\epsilon}}$ is compact, there exists $U^{\epsilon *}\in \Uadm_{LT^{\epsilon}}$ such that $V_{\alpha}^{LT^{\epsilon}} = \cJ_{\alpha, T^{\epsilon}}^{U^{\epsilon *}}$\,. Let, along a sub-sequence (denoting by same sequence) $U^{\epsilon *} \to \hat{U}^*$ as $\epsilon \to 0$\,, where  $\hat{U}^*\in \Uadm_{LI}$ (by Arzel\'a-Ascoli theorem)\,. Given that, $v_{\delta}^*$ is a Lipschitz continuous $\delta$-optimal control for the Ideal model, we get
\begin{align}\label{EforContvalue1B}
&\inf_{U\in \Uadm_{I}} \cJ_{\alpha, I}^{U} \leq   \cJ_{\alpha, I}^{\hat{U}^*} - \inf_{U^{\epsilon}\in \Uadm_{LT^{\epsilon}}} \cJ_{\alpha, T^{\epsilon}}^{U^{\epsilon}} + \inf_{U^{\epsilon}\in \Uadm_{LT^{\epsilon}}} \cJ_{\alpha, T^{\epsilon}}^{U^{\epsilon}}\nonumber \\
&\inf_{U^{\epsilon}\in \Uadm_{LT^{\epsilon}}} \cJ_{\alpha, T^{\epsilon}}^{U^{\epsilon}} \leq \cJ_{\alpha, T^{\epsilon}}^{v_{\delta}^{*}} - \inf_{U\in \Uadm_{I}} \cJ_{\alpha, I}^{U}  + \inf_{U\in \Uadm_{I}} \cJ_{\alpha, I}^{U}\,.
\end{align}
Thus, from (\ref{EforContvalue1B}), we deduce that
\begin{equation*}
\inf_{U\in \Uadm_{I}} \cJ_{\alpha, I}^{U} - \cJ_{\alpha, T^{\epsilon}}^{v_{\delta}^{*}} \leq \inf_{U\in \Uadm_{I}} \cJ_{\alpha, I}^{U} - \inf_{U^{\epsilon}\in \Uadm_{LT^{\epsilon}}} \cJ_{\alpha, T^{\epsilon}}^{U^{\epsilon}} \leq \cJ_{\alpha, I}^{\hat{U}^*} - \inf_{U^{\epsilon}\in \Uadm_{LT^{\epsilon}}} \cJ_{\alpha, T^{\epsilon}}^{U^{\epsilon}}\,. 
\end{equation*}  
Using the near optimality of $v_{\delta}^{*}$ and Proposition \ref{EforFixContConti1A}, it is easy to see that $-\delta \leq 
\inf_{U\in \Uadm_{I}} \cJ_{\alpha, I}^{U} - \cJ_{\alpha, T^{\epsilon}}^{v_{\delta}^{*}}$ as $\epsilon \to 0$. Also, from Proposition \ref{EForContValue1A}, it follows that $ \cJ_{\alpha, I}^{\hat{U}^*} - \inf_{U^{\epsilon}\in \Uadm_{LT^{\epsilon}}} \cJ_{\alpha, T^{\epsilon}}^{U^{\epsilon}} \to 0$ as $\epsilon\to 0$\,. This completes the proof of the theorem\,.
\end{proof}

Now using the above convergence result we want to show that any smooth near optimal policy of the idealized model is also near optimal in the true model.
\begin{theorem}\label{ThmNearOpt1A}
Suppose Assumptions (A1)-(A4) hold. Then for any Lipschitz continuous near optimal control $v_{\delta}^{*} \in \Usm$ of the idealized model (i.e., controlled diffusion model), we have 
\begin{equation*} 
\cJ_{\alpha, LT^{\epsilon}}^{v_{\delta}^*}(x, c) \leq V_{\alpha}^{LT^{\epsilon}} + 2\delta\,. 
\end{equation*}
\end{theorem}
\begin{proof}
By triangular inequality, it follows that
\begin{equation}\label{EThmNearOpt1B}
|V_{\alpha}^{LT^{\epsilon}} - \cJ_{\alpha, T^{\epsilon}}^{v_{\delta}^*}(x, c)| \leq  |V_{\alpha}^{LT^{\epsilon}} - V_{\alpha}^{I}| + |V_{\alpha}^{I} - \cJ_{\alpha, I}^{v_{\delta}^*}(x, c)| + | \cJ_{\alpha, I}^{v_{\delta}^*}(x, c) - \cJ_{\alpha, T^{\epsilon}}^{v_{\delta}^*}(x, c)|\,.   
\end{equation} From Theorem~\ref{EthmCont1A}, we have $|V_{\alpha}^{LT^{\epsilon}} - V_{\alpha}^{I}| \leq \delta$ as $\epsilon \to 0$, and by near optimality of $v_{\delta}^{*}$, we have $|V_{\alpha}^{I} - \cJ_{\alpha, I}^{v_{\delta}^*}(x, c)| \leq \delta$\,. Also, from Proposition~\ref{EforFixContConti1A}, we have $| \cJ_{\alpha, I}^{v_{\delta}^*}(x, c) - \cJ_{\alpha, T^{\epsilon}}^{v_{\delta}^*}(x, c)|\to 0$ as $\epsilon \to 0$\,. Thus, from (\ref{EThmNearOpt1B}), we conclude that
\begin{equation*} 
\cJ_{\alpha, LT^{\epsilon}}^{v_{\delta}^*}(x, c) \leq V_{\alpha}^{LT^{\epsilon}} + 2\delta\,. 
\end{equation*} This completes the proof\,.
\end{proof}
In the next section we present the robustness result for the finite horizon cost criterion\,.
\subsubsection{Analysis of the Finite Horizon Cost}\label{AnalysisFinite}
For some positive constant $\hat{M}_1$, let $$\hat{F}_L \df \{h:[0, T]\times \Rd \to U\mid |h(t_1,x) - h(t_2, y)| \leq \hat{M}_1\left(|t_1 - t_2| + |x - y|\right)\,\,\, \text{for all}\,\,\, x,y \in \Rd,\,\, t_1, t_2\in [0, T]\}\,.$$
Let the class of control policies induced by the functions in $\hat{F}_L$ by $\hat{\Uadm}_{LT^{\epsilon}} \df \{U_t\mid U_t = h(t, X^{\epsilon}(t)) \,\,\,\text{for some} \,\, h\in \hat{F}_L\}$ for true model and similarly $\hat{\Uadm}_{LI}$ for ideal model\,. By Arzel\'a- Ascoli theorem, we have that the set $\hat{\Uadm}_{LT^{\epsilon}}$ is compact for each $\epsilon >0$\,. If we consider the minimization problem (\ref{EFCost1OptAprox}) over the class of policies $\hat{\Uadm}_{LT^{\epsilon}}$, the associated optimal values are defined by $\cJ_{T, LT^{\epsilon}}^{*}(x,c)\,\df\, \inf_{U\in \hat{\Uadm}_{LT^{\epsilon}}}\cJ_{T, T^{\epsilon}}^{U}(x, c)$ (optimal finite horizon cost)\,.
Next Lemma shows that for each fixed Lipschitz continuous policy the finite horizon cost is continuous with respect to the model perturbation.
\begin{lemma}\label{EFinforFixContConti1A}
Suppose Assumptions (A1)-(A4) hold. Then for any Lipschitz continuous policy $v = h(\cdot, \cdot)$ (where $h\in \hat{F}_L$), we have 
\begin{align*}
 \cJ_{T, T^{\epsilon}}^{v} \to  \cJ_{T, I}^{v}\quad\text{as}\quad \epsilon\to 0\,.
\end{align*}
\end{lemma}
\begin{proof}
Since $\upsigma\in \cC_b^3(\Rd)$ and for any Lipschitz continuous $v= h(\cdot, \cdot)$, we have $b(\cdot, h(\cdot, \cdot))$ is Lipschitz continuous, thus from Theorem~\ref{ThmCont1A} we have that the solution $X^{\epsilon}$ of (\ref{ETM}) converges to the solution $X$ of (\ref{EAM}) almost surely in sup norm as $\epsilon \to 0$\,. Thus by generalized dominated convergence theorem (see \cite[Theorem 3.5]{serfozo1982convergence}) we obtain our result\,.    
\end{proof}

In the following lemma, we prove that for any $U^{\epsilon}\in \hat{\Uadm}_{LT^{\epsilon}}$ with $U^{\epsilon}\to \hat{U}$, we have associated costs also converge as $\epsilon\to 0$\,. 
\begin{lemma}\label{EFinForContValue1A}
Suppose Assumptions (A1)-(A4) hold\,. Then for any sequence of policy $U^{\epsilon}\in \hat{\Uadm}_{LT^{\epsilon}}$ with $U^{\epsilon}\to \hat{U}$, we have 
\begin{align*}
 \cJ_{T, T^{\epsilon}}^{U^{\epsilon}} \to  \cJ_{T, I}^{\hat{U}}\quad\text{as}\quad \epsilon\to 0,  
\end{align*} where $\hat{U}\in\hat{\Uadm}_{LI}$\,.
\end{lemma}
\begin{proof} 
Since $U^{\epsilon}\in \hat{\Uadm}_{LT^{\epsilon}}$, we have $U^{\epsilon}_t = h^{\epsilon}(t, X_t^{\epsilon})$ for some $h^{\epsilon}\in \hat{F}_L$\,. By Arzel\'a-Ascoli theorem, we have that $h^{\epsilon}(\cdot, \cdot) \to \hat{h}(\cdot, \cdot)$ uniformly over compacts, for some $\hat{h}\in \hat{F}_L$. In particular we have the limit $\hat{U}_t = \hat{h}(t, X_t)$\,. This implies that $\hat{U}\in\hat{\Uadm}_{LI}$\,. Thus, from Assumption (A1), it is easy to see that $b(\cdot, h^{\epsilon}(\cdot, \cdot))$ and $b(\cdot, \hat{h}(\cdot,\cdot))$ are Lipschitz continuous functions. Moreover, we have $b(\cdot, h^{\epsilon}(\cdot,\cdot)) - b(\cdot, \hat{h}(\cdot,\cdot))$ is uniformly Lipschitz continuous and converges to zero uniformly over compacts\,. Now, from Theorem~\ref{ThmCont1A}, we get that the solution $X^{\epsilon}$ of (\ref{ETM}) corresponding to $U^{\epsilon}$ converges to the solution $X$ of (\ref{EAM}) corresponding to $\hat{U}$ in $\|\cdot\|_\infty$ norm almost surely as $\epsilon \to 0$\,. Thus, by generalized dominated convergence theorem as in \cite[Theorem 3.5]{serfozo1982convergence}, we obtain our result\,.     
\end{proof}

Using Lemma~\ref{EFinForContValue1A}, we want to prove that in the limit the value functions are close to each other\,. 
\begin{theorem}\label{EFinthmCont1A}
Suppose Assumptions (A1)-(A4) hold\,. If for any $\delta>0$, $v^{\delta}$ is a Lipschitz continuous $\delta$-optimal policy of the idealized model, i.e.,
\begin{equation*}
   \cJ_{T, I}^{v^{\delta}} \leq \cJ_{T,I}^{*} + \delta\,,
\end{equation*} with $|v^{\delta}(t, x) - v^{\delta}(s, y)| \leq M_2 \left(|t-s| + |x-y|\right)$ for some $0< M_2 \leq \hat{M}_1$\,.
Then, we have
\begin{equation*}
  \lim_{\epsilon \to 0} |\cJ_{T,LT^{\epsilon}}^{*} - \cJ_{T, I}^{*}| \leq \delta\,.  
\end{equation*}
\end{theorem}
\begin{proof}
In view of Theorem~\ref{ThmCont1A}, it is easy to see that (by the generalized dominated convergence theorem) for each $\epsilon > 0$, the map $U^{\epsilon}\to \cJ_{T, T^{\epsilon}}^{U^{\epsilon}}$ is continuous on $\hat{\Uadm}_{LT^{\epsilon}}$\,. Also, given that $\hat{\Uadm}_{LT^{\epsilon}}$ is compact, we have there exists $\hat{U}^{\epsilon *}\in \hat{\Uadm}_{LT^{\epsilon}}$ such that $\cJ_{T,T^{\epsilon}}^{\hat{U}^{\epsilon *}} = \cJ_{T, LT^{\epsilon}}^*$\,. By Arzel\'a-Ascoli, along a sub-sequence (denoting by same sequence) we have $\hat{U}^{\epsilon *} \to \bar{U}^*$ as $\epsilon \to 0$\,, where  $\bar{U}^*\in \Uadm_{LI}$\,. As we know that $v^{\delta}$ is a Lipschitz continuous $\delta$-optimal control for Ideal model, we obtain
\begin{align}\label{EFinforContvalue1B}
&\inf_{U\in \Uadm_{I}} \cJ_{T, I}^{U} \leq   \cJ_{T, I}^{\bar{U}^*} - \inf_{U^{\epsilon}\in \hat{\Uadm}_{LT^{\epsilon}}} \cJ_{T, T^{\epsilon}}^{U^{\epsilon}} + \inf_{U^{\epsilon}\in \hat{\Uadm}_{LT^{\epsilon}}} \cJ_{T, T^{\epsilon}}^{U^{\epsilon}}\nonumber \\
&\inf_{U^{\epsilon}\in \hat{\Uadm}_{LT^{\epsilon}}} \cJ_{T, T^{\epsilon}}^{U^{\epsilon}} \leq \cJ_{T, T^{\epsilon}}^{v^{\delta}} - \inf_{U\in \Uadm_{I}} \cJ_{T, I}^{U}  + \inf_{U\in \Uadm_{I}} \cJ_{T, I}^{U}\,.
\end{align}
Now, from (\ref{EFinforContvalue1B}), it follows that
\begin{equation*}
\inf_{U\in \Uadm_{I}} \cJ_{T, I}^{U} - \cJ_{T, T^{\epsilon}}^{v^{\delta}} \leq \inf_{U\in \Uadm_{I}} \cJ_{T, I}^{U} - \inf_{U^{\epsilon}\in \hat{\Uadm}_{LT^{\epsilon}}} \cJ_{T, T^{\epsilon}}^{U^{\epsilon}} \leq \cJ_{T, I}^{\bar{U}^*} - \inf_{U^{\epsilon}\in \hat{\Uadm}_{LT^{\epsilon}}} \cJ_{T, T^{\epsilon}}^{U^{\epsilon}}\,. 
\end{equation*}  
In view of near optimality of $v^{\delta}$ and Lemma~\ref{EFinforFixContConti1A}, we obtain $-\delta \leq 
\inf_{U\in \Uadm} \cJ_{T, I}^{U} - \cJ_{T, T^{\epsilon}}^{v^{\delta}}$ as $\epsilon \to 0$. From Lemma~\ref{EFinForContValue1A}, we have $ \cJ_{T, I}^{\bar{U}^*} - \cJ_{T, T^{\epsilon}}^{\hat{U}^{\epsilon *}} \to 0$ as $\epsilon\to 0$\,. Since $\inf_{U^{\epsilon}\in \hat{\Uadm}_{LT^{\epsilon}}} \cJ_{T, T^{\epsilon}}^{U^{\epsilon}} = \cJ_{T, T^{\epsilon}}^{\hat{U}^{\epsilon *}}$, this completes the proof of the theorem\,.
\end{proof}

Next we want to show that any smooth $\delta$-optimal policy in idealized model is $2\delta$-optimal in true model for small enough $\epsilon$\,.
\begin{theorem}\label{ThmFinNearOpt1A}
Suppose Assumptions (A1)-(A4) hold. Then for any Lipschitz continuous $\delta$- optimal control $v^{\delta}$ of the idealized model, we have 
\begin{equation*} 
\cJ_{T, T^{\epsilon}}^{v^{\delta}}(x, c) \leq \cJ_{T, LT^{\epsilon}}^{*} + 2\delta\,,
\end{equation*} for small enough $\epsilon$\,.
\end{theorem}
\begin{proof}
By triangular inequality, we deduce that
\begin{equation}\label{EFinThmNearOpt1B}
|\cJ_{T, LT^{\epsilon}}^*(x, c) - \cJ_{T, T^{\epsilon}}^{v^{\delta}}(x, c)| \leq  |\cJ_{T, LT^{\epsilon}}^*(x, c) - \cJ_{T, I}^{*}(x, c)| + |\cJ_{T, I}^{*}(x, c) - \cJ_{T, I}^{v^{\delta}}(x, c)| + | \cJ_{T, I}^{v^{\delta}}(x, c) - \cJ_{T, T^{\epsilon}}^{v^{\delta}}(x, c)|\,.   
\end{equation} From Theorem~\ref{EFinthmCont1A}, we have $|\cJ_{T, LT^{\epsilon}}^*(x, c) - \cJ_{T, I}^{*}(x, c)| \leq \delta$ as $\epsilon \to 0$, and by $\delta$- optimality of $v^{\delta}$, we have $|\cJ_{T, I}^{*}(x, c) - \cJ_{T, I}^{v^{\delta}}(x, c)| \leq \delta$\,. Also, from Lemma~\ref{EFinforFixContConti1A}, we have $| \cJ_{T, I}^{v^{\delta}}(x, c) - \cJ_{T, T^{\epsilon}}^{v^{\delta}}(x, c)|\to 0$ as $\epsilon \to 0$\,. This completes the proof of the theorem\,.
\end{proof}

\subsection{Application: Systems driven by Wide-Band Noise and Approximations to the Brownian}

In this section, we present the near optimality of smooth policies in true models which are driven by Karhunen-Lo\`eve Wong-Zakai, mollified Brownian or fractional Brownian approximation of the Brownian noise\,. In particular, our true model is
\begin{equation}\label{ETMEx}
d X^{n}(t) = b(X^{n}(t), U(t))dt + \sigma(X^{n}(t))d\mathbf W^{n}(t)  
\end{equation} where $W^{n}$ is one of the above approximations of the Brownian noise\,, and the idealized model is given by (\ref{EAM})\,.

Then, from Corollary~\ref{cor:approximation-Strat}, it is easy to see that under Lipschitz continuous policies, the solution $X^n$ of (\ref{ETMEx})converges to the solution $X$ of (\ref{EAM}) in sup norm almost surely as $n\to \infty$\,. Now, following analysis as in Subsection~\ref{AnalysisDis}, we have the following near optimality result for the discounted cost case.
\begin{theorem}\label{ThmExaNearOptDis}
Suppose Assumptions (A1)-(A4) hold. Then for any Lipschitz continuous $\delta$- optimal control $v_{\delta}^*$ of the idealized model, we have 
\begin{equation*} 
\cJ_{\alpha, T^n}^{v_{\delta}^*}(x, c) \leq \cJ_{\alpha, LT^{n}}^{*} + 2\delta\,,
\end{equation*} for large enough $n$\,.    
\end{theorem}
Similarly, for the finite horizon case, following the analysis in Subsection~\ref{AnalysisFinite}, we have the following near optimality result\,.
\begin{theorem}\label{ThmExaNearOptFinite}
Suppose Assumptions (A1)-(A4) hold. Then for any Lipschitz continuous $\delta$- optimal control $v^{\delta}$ of the idealized model, we have 
\begin{equation*} 
\cJ_{T, T^n}^{v^{\delta}}(x, c) \leq \cJ_{T, LT^{n}}^{*} + 2\delta\,,
\end{equation*} for large enough $n$\,.    
\end{theorem}

\subsection{Beyond Lipschitz and Markov Controls}
We highlight a subtle aspect on the space of control policies considered in the analysis so far. In the paper, we considered Lipschitz controls in the state variable, which may not be optimal for the case with non-Brownian noise. However, for the standard theory on optimal control of diffusions, compactness-continuity analysis for the existence and approximation of optimal policies is greatly facilitated (and enriched in view of the classes of policies considered) by the relaxation of control policies in such a way that the policies are (i)
${\cal P}(\mathbb{U})$-valued, where  ${\cal P}(\mathbb{U})$ denotes the space of probability measures on $\mathbb{U}$ under the weak convergence topology (i.e., policies are randomized), and (ii) they are {\it wide-sense admissible}, meaning that the control and state processes are independent of future increments of the driving noise processes; that is the control action variables $\{U_t\}$ satisfy the following non-anticipativity condition: for $s<t\,,$ $B_t - B_s$ is independent of
$$\sF_s := \,\,\mbox{the completion of}\,\,\, \sigma(X_0, U_r, B_r : r\leq s)\,\,\,\mbox{relative to} \,\, (\sF, \mathbb{P})\,.$$ 

Such a machinery leads to very general existence results for controlled diffusions under a variety of information structures; see e.g. \cite{arapostathis2012ergodic, FlPa82, kushner1990numerical,kushner2001numerical,kushner2012weak, pradhanyuksel2023DTApprx}.

However, when the noise is not Brownian, since the future increments are possibly correlated with earlier noise realizations, the wide-sense admissible framework is not applicable for such noise processes. Furthermore, if one considers adapted policies (policies causally measurable on the noise itself), such policy spaces are known to be non-closed under weak convergence \cite{arapostathis2012ergodic}. Accordingly, we considered only Markovian policies in our analysis; which, however, turned out to lead to very general robustness properties in view of the classes of noise processes considered.

Nonetheless, for a class of near-Brownian noise processes, one can consider the relaxed control framework; in particular for the Wong-Zakai approximation, the relaxed formulation is applicable where we can consider policies which are strictly adapted to the driving noise process. In this case, it can be shown that as the process measure on $(X^{\epsilon}_s,U_s,\xi^{\epsilon})$ converges weakly, the limit process will be so that the control policy is wide-sense admissible provided that $P(d\xi^{\epsilon}_{s}|\xi^{\epsilon}_{t})$ converges to the Brownian $P(dB_{s})$ continuously (see the proof of \cite[Theorem 5.6]{YukselWitsenStandardArXiv}): if we have continuous convergence of $P(d \xi_{s}^{\epsilon}| \xi_{t}^{\epsilon})$ to $P(d B_{s})$ for any $s> t$, we have that the space of policies is closed under weak convergence\,. Now, for any sequence of policy $U^{\epsilon}\in \Uadm_{T^{\epsilon}}$ with $U^{\epsilon}\to U$, we have $U\in \Uadm_{I}$. By Skorohod's theorem, there exists a probability space in which $U^{\epsilon}(\omega) \to U(\omega)$ weakly as a $\cP(\Act \times [0, \infty))$-valued probability measure, almost surely (in $\omega$). Moreover, since $b(\cdot, U_t^{\epsilon}) = \int_{\Act}b(\cdot, \zeta)U_t^{\epsilon}(d \zeta) $ and $b(\cdot, U_t) = \int_{\Act}b(\cdot, \zeta)U_t(d \zeta)$ are uniformly Lipschitz continuous (in particular $b(\cdot, U_t^{\epsilon})$ and $b(\cdot, U_t)$ are equicontinuous and equibounded), from Theorem~\ref{ThmCont1A}, we have that the solution $X^{\epsilon}$ of (\ref{ETM}) associated to $U^{\epsilon}$ converge to the solution $X$ of (\ref{EAM}) associated to $U$. Thus, for discounted cost case: following steps as in the Subsection~\ref{AnalysisDis}, we obtain (\ref{EdisNearOpt}), and similarly for the discounted cost case following the steps as in the Subsection~\ref{AnalysisFinite} we get (\ref{EFiniteHNearOpt})\,.      


\section{Conclusion}

 In this article we presented a robustness theorem for controlled stochastic differential equations driven by approximations of Brownian motion. We showed that within the class of Lipschitz continuous control policies which are also shown to be near optimal for the idealized Brownian models among all admissible policies, an optimal solution for the Brownian idealized model is near optimal for a true system driven by a non-Brownian (but near-Brownian) noise. We note that only considering Lipschitz policies is the price we pay for robustness. 

 \subsection{Robustness to Interpretations on Integration Theories}
We also conclude with one further, perhaps more philosophical, implication of our analysis. Sometimes, it is not the noise that is incorrectly modelled but the integration theory itself that is incorrectly modelled. There are a myriad of different integration theories such as the classical It\^o and Stratonovich (and many others as noted in \cite{McShane-Belated}\cite{Escudero-1,Escudero-2}). A priori there is no reason to prefer one integration theory over another - the dynamics of the system sometimes determines which integration theory serves as the best model. Rough paths theory can also be viewed as a way of ``parameterizing" integration theories. That is, given a rough path $\mathbf X=(X,\mathbb X)$ and a function $F\in C^{2\alpha}$ the object $\widetilde{\mathbf X_{s,t}}:=(X_{s,t},\widetilde{\mathbb X_{s,t}}):=(X_{s,t},\mathbb X_{s,t}+F(t)-F(s))$ is also a rough path above $X$. Conversely, every rough path above the signal $X$ differs by the increment of some function $F\in C^{2\alpha}$. For example $\mathbf W^{Strat}=(W,\mathbb W^{Strat})=(W,\mathbb W^{Ito}+\frac{1}{2}(t-s)I),$ where $I$ is the identity matrix. 

Given two rough paths above Brownian motion, $\mathbf W_{s,t}=(W_{s,t},\mathbb W_{s,t})$ and $\widetilde{\mathbf W_{s,t}}=(W_{s,t},\mathbb W_{s,t}+F(t)-F(s))$, their rough path distance is just the $2\alpha$ H\"older norm $\|F\|_{2\alpha}$. Our analysis can handle a sequence of rough paths $\mathbf W^n$ converging to $\mathbf W^{Strat}$, where the integration theory itself is changing.

\bibliographystyle{plain}
\bibliography{bibliography,Somnath_Robustness,SerdarBibliography}

\end{document}